\documentclass{article}[12pt]

\usepackage{fullpage}
\usepackage{amssymb}
\usepackage{amsmath}
\usepackage{amsthm}
\usepackage{url}
\usepackage{hyperref}
\usepackage{mathabx}
\usepackage{multirow}
\usepackage{pdflscape}
\usepackage{tablefootnote}

\newtheorem{theorem}{Theorem}
\newtheorem{corollary}{Corollary}

\newtheorem{lemma}{Lemma}

\newtheorem{definition}{Definition}

\newcommand{\ijktuple}{(I, +_I, \preceq_I, J, +_J, \preceq_J, K, +_K, \preceq_K, \ast)}
\newcommand{\ituple}{(I, +, \preceq_I, \ast)}

\begin{document}

\title{On rearrangement inequalities for multiple sequences}
\date{February 24, 2020\\ Latest update: April 14, 2022}
\author{Chai Wah Wu\\ IBM T. J. Watson Research Center\\ P. O. Box 218, Yorktown Heights, New York 10598, USA\\e-mail: chaiwahwu@ieee.org}

\maketitle

\begin{abstract}
The classical rearrangement inequality provides bounds for the sum of products of two sequences under permutations of terms and show that similarly ordered sequences provide the largest value whereas opposite ordered sequences provide the smallest value. This has been generalized to multiple sequences to show that similarly ordered sequences provide the largest value. However, the permutations of the sequences that result in the smallest value are generally not known. We show a variant of the rearrangement inequality for which a lower bound can be obtained and conditions for which this bound is achieved for a sequence of permutations. We also study a generalization of the rearrangement inequality and a variation where the permutations of terms can be across the various sequences. For this variation, we can also find the minimizing and maximizing sequences under certain conditions. Finally, we also look at rearrangement inequalities of other objects that can be ordered such as functions and matrices.
\end{abstract}

\section{Introduction}
The rearrangement inequality \cite{Hardy1952} states that given two finite sequences of real numbers the sum of the product of pairs of terms is maximal when the sequences are similarly ordered and minimal when oppositely ordered. More precisely, suppose $x_1 \leq x_2 \cdots \leq x_n$ 
and $y_1 \leq y_2 \cdots \leq y_n$, then for any permutation $\sigma$ in the symmetric group $S_n$ of permutations on $\{1,\cdots , n\}$,
\begin{equation}x_ny_1+\cdots +x_1y_n \leq x_{\sigma(1)}y_1+\cdots + x_{\sigma(n)}y_n \leq x_1y_1+\cdots x_ny_n
\label{eqn:rearrange1}
\end{equation}

The dual inequality is also true \cite{oppenheim:rearrangement:1954}, albeit only for nonnegative numbers in general (i.e. $x_i\geq 0$, $y_i\geq 0$):

\begin{equation}
(x_1+y_1) \cdots  (x_n+y_n) \leq (x_{\sigma(1)}+y_1)\cdots  (x_{\sigma(n)}+y_n ) \leq (x_n+y_1)\cdots (x_1+y_n)
\label{eqn:rearrange2}
\end{equation}

Eq. (\ref{eqn:rearrange2}) says that similarly ordered terms minimize the product of sums of pairs, while opposite ordered terms maximize the product of sums.
In Ref. \cite{Minc1971} it was shown that Eq. (\ref{eqn:rearrange1}) and Eq. (\ref{eqn:rearrange2}) are equivalent for positive numbers.

In Ref. \cite{Ruderman1952}, these inequalities are generalized to multiple sequences of numbers:
\begin{lemma}\label{lem:ruderman}
Consider a set of nonnegative numbers $\{a_{ij}\}$, $i=1,\cdots, k$, $j=1,\cdots, n$.  For each $i$, let
$a'_{i1},a'_{i2},\cdots,a'_{in}$ be the numbers $a_{i1},a_{i2},\cdots,a_{in}$ reordered such that
$a'_{i1}\geq a'_{i2}\geq\cdots\geq a'_{in}$.  Then
\[ \sum_{j=1}^n\prod_{i=1}^k a_{ij} \leq   \sum_{j=1}^n\prod_{i=1}^k a'_{ij} \]
\[ \prod_{j=1}^n\sum_{i=1}^k a_{ij} \geq   \prod_{j=1}^n\sum_{i=1}^k a'_{ij} \]
\end{lemma}

Note that only half of the rearrangement inequality is generalized. In particular, the rightmost inequality (the upper bound) in Eq. (\ref{eqn:rearrange1}) and the leftmost inequality (the lower bound) in 
Eq. (\ref{eqn:rearrange2}) are generalized in Lemma \ref{lem:ruderman} by showing that similarly ordered sequences maximizes the sum of products and minimizes the product of sums. No such generalization is known for the other half. This paper provides results for the other direction and generalizes the rearrangement inequalities in various ways.

Eq. (\ref{eqn:rearrange1}) can be used to prove the AM-GM inequality which states that the algebraic mean of nonnegative numbers are larger than or equal to their geometric mean. We will rewrite it in the following equivalent form.

\begin{lemma}[AM-GM inequality]\label{lem:am-gm}
For $n$ nonnegative real numbers $x_i$, $\sum_{i=1}^n x_i \geq n \sqrt[n]{\prod_{i=1}^n x_i}$ and   $\prod_{i=1}^n x_i \leq \left(\frac{\sum_{i=1}^n x_i}{n}\right)^n$ with equality if and only if all the $x_i$ are the same.
\end{lemma}

This allows us to give the following bounds on the other direction of Lemma \ref{lem:ruderman}.
\begin{lemma}
Consider a set of nonnegative numbers $\{a_{ij}\}$, $i=1,\cdots, k$, $j=1,\cdots, n$. Then
\[ n\sqrt[n]{\prod_{ij}a_{ij}} \leq \sum_{j=1}^n\prod_{i=1}^k a_{ij} \]
\[ \left(\frac{\sum_{ij}a_{ij}}{n}\right)^n \geq \prod_{j=1}^n\sum_{i=1}^k a_{ij} \]
\end{lemma}

In addition, Lemma \ref{lem:am-gm} implies that if there exists $k$ permutations $\sigma_i$ on $\{1,\cdots ,n\}$ such that $\prod_{i=1}^k a_{i\sigma_i(j)} = \prod_{i=1}^k a_{i\sigma_i(1)}$ for all $j$, then this set of permutations will achieve the lower bound and minimize the sum of products, i.e.
\[ \sum_{j=1}^n\prod_{i=1}^k a_{i\sigma_i(j)} \leq \sum_{j=1}^n\prod_{i=1}^k a_{ij} \]
Similarly, if there exists permutations $\sigma_i$ such that $\sum_{i=1}^k a_{i\sigma_i(j)} = \sum_{i=1}^k a_{i\sigma_i(1)}$ for all $j$, then this set of permutations will achieve the upper bound and maximize the product of sums, i.e.
\[ \prod_{j=1}^n\sum_{i=1}^k a_{i\sigma_i(j)} \geq \prod_{j=1}^n\sum_{i=1}^k a_{ij} \]
In the next section we consider scenarios where these conditions can be satisfied for some sequence of permutations of terms and thus supply the other directions of Lemma \ref{lem:ruderman}.

\section{Sums of products of permuted sequences}\label{sec:sumprod}
Instead of considering multiple sequences, we restrict ourselves to permutations of the same sequence and look at sum of products of these sequences.

\begin{definition}
Let $0 \leq a_1 \leq  a_2  \dots \leq a_n$ be a sequence of nonnegative numbers. 
Consider $k$ permutations of the integers $\{1,\cdots, n\}$ denoted as $\{\sigma_1,\cdots, \sigma_k\}$
and define the value 
$v(n,k) = \sum_{i=1}^n \prod_{j=1}^k a_{\sigma_j(i)}$.  The maximal and minimal value of $v$ among all $k$-sets of permutations are denoted as $v_{\max}(n,k)$ and $v_{\min}(n,k)$ respectively.
\end{definition}

An immediate consequence of Lemma \ref{lem:ruderman} is that
$v_{\max}(n,k) = \sum_{i=1}^n a_i^k$ and 
is achieved when 
all the $k$ permutations $\sigma_i$ are the same.  

$v_{\min}(n,k)$ and $v_{\max}(n,k)$ can be determined explicitly for small value of $n$ or $k$.
\begin{lemma}
\begin{itemize}
\item $v(1,k) = a_i^k$,
\item $v(n,1) = \sum_{i=1}^n a_i$,
\item $v_{\max}(2,k) = a_1^k+a_2^k$.
\item $v_{\min}(2,2m) = 2a_1^ma_2^m$
\item $v_{\min}(2,2m+1) = (a_1+a_2)a_1^ma_2^m$
\item $v_{\max}(n,2) = \sum_{i=1}^n a_i^2$
\item  $v_{\min}(n,2) = \sum_{i=1}^n a_ia_{n-i+1}$
\end{itemize}
\end{lemma}

\begin{proof}
For $k=1$ there is only one sequence and $v(n,1) = \sum_{i=1}^n a_i$. For $n=1$, the only permutation is $(1)$, so $v(1,k) = a_1^k$.
When $n = 2$, there are only two permutations on the integers $\{1,2\}$, and $v_{\max}(2,k) = a_1^k+a_2^k$.  If $k=2m$, $v_{\min}(2,k) = 2a_1^ma_2^m$ is achieved with $m$ of the permutations of one kind and the other half the other kind.
If $k=2m+1$, $v_{\min}(2,k) = (a_1+a_2)a_1^ma_2^m$ is achieved with $m$ of the permutations of one kind and $m+1$ of them the other kind.

The rearrangement inequality (Eq. (\ref{eqn:rearrange1})) implies that for $k=2$,  $v_{\max}(n,2) = \sum_{i=1}^n a_i^2$ and  $v_{\min}(n,2) = \sum_{i=1}^n a_ia_{n-i+1}$ by choosing both permutations to be  ($1$,$2$,$\cdots$, $n$) for $v_{\max}(n,2)$
and choosing the two permutations to be ($1$,$2$,$\cdots$, $n$) and ($n$,$n-1$,$\cdots$ , $2$, $1$) for $v_{\min}(n,2)$.  
\end{proof}

Our next result is a lower bound on $v_{\min}$:

\begin{lemma} \label{lem:vminbound}
$v_{\min}(n,k) \geq  n \prod_i a_i^{k/n}$.
\end{lemma}
\begin{proof}
The product $\prod_{ij} a_{\sigma_i(j)}$ is equal to $\prod_i a_i^k$. Thus by Lemma \ref{lem:am-gm}, $v(n,k) \geq n\sqrt[n]{ \prod_i a_i^k} = n \prod_i a_i^{k/n}$.
\end{proof}

Our main result in this section is that this bound is tight when $k$ is a multiple of $n$.

\begin{theorem}
If $n$ divides $k$, then $v_{min}(n,k) = n\prod_{i=1}^n a_i^{k/n}$ and is achieved by using each cyclic permutation $k/n$ times..
\end{theorem}

\begin{proof}
By Lemma \ref{lem:vminbound} $v(n,k) \geq  n \prod_{i=1}^n a_i^{k/n}$.
Consider the $n$ cyclic permutations $r_1 = (1,2,...,n)$, $r_2 = (2,...,n,1)$, ..., $r_n = (n,1,...,n-1)$.
It is clear that using $k/n$ copies of each permutation $r_i$ to form $k$ permutations results in $v(n,k) = n\prod_{i=1}^n a_i^{k/n}$.
\end{proof}

\section{The dual problem of product of sums}\label{sec:prodsum}

\begin{definition}
Let $0 \leq a_1 \leq  a_2  \dots \leq a_n$ be a sequence of nonnegative numbers. 
Consider $k$ permutations of the integers $\{1,\cdots, n\}$ denoted as $\{\sigma_1,\cdots, \sigma_k\}$
and define the value 
$w(n,k) = \prod_{i=1}^n \sum_{j=1}^k a_{\sigma_j(i)}$.  The maximal and minimal value of $v$ among all $k$-sets of permutations are denoted as $w_{\max}(n,k)$ and $w_{\min}(n,k)$ respectively\footnote{To reduce the amount of notation, $v$, $w$, $v_{\min}$, $v_{\max}$,
$w_{\min}$, $w_{\max}$ are redefined in various subsections and the results about them are valid within the subsection.}.
\end{definition}

Analogous to Section \ref{sec:sumprod} the following results can be derived regarding $w_{\max}$ and $w_{\min}$.

\begin{lemma}\label{lem:ubwmax}
\begin{itemize}
\item $w_{\min}(n,k) = \prod_{i=1}^n k a_i = k^n \prod_i a_i $ 
\item $w_{\max}(1,k) = ka_1$
\item $w_{\max}(n,1) = \prod_i a_i$
\item $w_{\min}(2,k) = k^2 \prod_i a_i$.
\item $w_{\max}(2,2m) = (a_1+a_2)^2m^2$.
\item $w_{\max}(2,2m+1) = (ma_1+(m+1)a_2)(ma_2+(m+1)a_1)$.
\item $w_{\min}(n,2) = 2^n \prod_i a_i$.
\item $w_{\max}(n,2) = \prod_i (a_i+a_{n-i+1})$.
\item $w_{\max}(n,k) \leq \left(\frac{k\sum_i a_i}{n}\right)^n$ with equality if $n$ divides $k$.
\end{itemize}
\end{lemma}

\section{The special case where $a_i$ is an arithmetic progression}\label{sec:ap}
Consider the special case where the elements $a_i$ form an arithmetic progression, i.e. $a_i$ are equally spaced where $a_{i+1} - a_i$ is constant and does not depend on $i$.
Even though $v_{\min}$ are difficult to compute in general, explicit forms for $w_{\max}$ can be found for many values of $n$ and $k$.

\begin{theorem} \label{thm:k=even}
If $k = 2t+nu$ for nonnegative integers $t$ and $u$, then $w_{\max}(n,k) = \left(\frac{k(a_1+a_n)}{2}\right)^n$.
\end{theorem}

\begin{proof}
It is easy to see that $\sum_i a_i = n(a_1+a_n)/2$. By Lemma \ref{lem:ubwmax} $w_{\max}(n,k) \leq \left(\frac{k(a_1+a_n)}{2}\right)^n$. 
By using $t$ copies of the permutation 
$(1,\cdots, n)$ and $t$ copies of the permutation $(n,\cdots, 1)$ followed by $u$ copies each of the cyclic permutations $r_i$,  we see that $\sum_{j}\sigma_j(i) = t(a_1+a_n)+un(a_1+a_n)/2 = (t+un/2)(a_1+a_n) = k(a_1+a_n)/2$ for all $i$ and thus $w(n,k) = \left(\frac{k(a_1+a_n)}{2}\right)^n$. 
\end{proof}

\begin{corollary}
If $k$ is even, then $w_{\max}(n,k) = \left(\frac{k(a_1+a_n)}{2}\right)^n$.
\end{corollary}

\begin{corollary}
If $n$ is odd and $k \geq n-1$, then $w_{\max}(n,k) = \left(\frac{k(a_1+a_n)}{2}\right)^n$.
\end{corollary}

The case when $k$ is odd and $n$ is even is more involved.
Let $a_i = a_1 + (i-1)d = (a_1-d) + id$ for $i = 1,\cdots, n$ and $d\geq 0$.
Given a $k$-set of permutations $\sigma_j$ define $w_i$ as $w_i = \sum_{j=1}^k \sigma_j(i) $. This implies that $\sum_{j=1}^k a_{\sigma_j(i)} = k(a_1-d) + w_i d$.
Next we show there is a sequence of permutations for which $w_i-w_j \leq 1$ for all $i,j$ when $k\geq n-1$.

\begin{lemma} \label{lem:n-1seq}
If $n$ is even, there exists a sequence $\sigma_j$ of $n-1$ permutations of $\{1,\cdots n\}$ such that $w_i = \frac{n^2}{2}-1$ for $i=1,\cdots \frac{n}{2}$ and
$w_i = \frac{n^2}{2}$ for $i=\frac{n}{2}+1,\cdots, n$. 
\end{lemma}
\begin{proof}
Recall the cyclic permutations denoted as $r_i$. Consider the index set $S = \{i: 2\leq i \leq n, i \neq n/2+1\}$.
Let us compute $\sum_{j\in S}r_j(i)$. Since $r_1(i) = (1,2,...,n)$ and $r_{n/2+1} = (n/2+1,n/2+2,...,n/2)$,
 $\sum_{j\in S}^{n-1}r_j(i) = n(n+1)/2 - r_1(i) - r_{n/2+1}(i)$ is equal to
 $ n(n+1)/2- i - (n/2+i) = n^2/2-2i$ for $i = 1,\cdots, n/2$ and equal to   $ n(n+1)/2- i - (i-n/2) = n^2/2-(2i-n)$ for $i = n/2+1,\cdots, n$.
 Let $\tilde{\sigma}$ be the permutation defined as $\tilde{\sigma}(i) = 2i-1$ for $i = 1\cdots n/2$ and $\tilde{\sigma}(i) = n-2i$ for $i=n/2+1\cdots, n$.
  Define the $(n-1)$-set of permutations $\{\sigma_i$\} as $\tilde{\sigma}$ plus the cyclic permutations with index in $S$, we get
 $\sum_{j=1}^{n-1} \sigma_j(i) = n^2/2-1$ for $i = 1,\cdots, n/2$ and $\sum_j \sigma_j(i) = n^2/2$ for $i=n/2+1,\dots, n$. 
\end{proof}

\begin{corollary}
If $n$ is even and $k$ is odd, there does not exists a $k$-set of permutations such that $w_i = w_j$ for all $i,j$.
If $k\geq n-1$, then there exists $k$ permutations such that $w_i-w_j\leq 1$ for all $i,j$.
\end{corollary}
\begin{proof}
If $n$ is even and $k$ is odd, $\sum_i w_i =kn(n+1)/2$ is not divisible by $n$ as $k$ and $n+1$ are both odd. This means it is not possible for $w_i = w_j$ for all $i,j$. 
If $n$ is odd, the case $k=n-1$ can be achieved with $k/2$ permutations $(1,\cdots , n)$ and $k/2$ permutations $(n,n-1,\dots ,1)$.
If $n$ is even, the case $k=n-1$ follows from Lemma \ref{lem:n-1seq}.
If $k > n$, it follows by induction from the $k-2$ case and adding the two permutations $(1,\cdots , n)$ and  $(n,n-1,\cdots ,1)$.
\end{proof}

\begin{lemma}\label{lem:max2}
If $w_1+w_2 = v_1+v_2$ and $|w_2 - w_1| \geq |v_2 - v_1|$, then
$(x + w_1)(x+w_2) \leq (x+v_1)(x+v_2)$.
\end{lemma}
\begin{proof}
Let $y = w_1+w_2$. Then
$(x + w_1)(x+w_2) = x^2 + yx + w_1(y-w_1)$. Since the function $x(y-x)$ has a maximum at $\frac{y}{2}$, this implies that $(x + w_1)(x+w_2)$ is maximized when $w_1 = w_2$.
\end{proof}

\begin{lemma}\label{lem:maxperm}
If $k \geq n-1$, then for the set permutations $\sigma_j$ that maximizes $w(n,k)$, the corresponding $w_i$ must satisfy $w_i-w_j \leq 1$ for all $i,j$.
If in addition, $n$ is odd or $k$ is even, then $w_i = w_j$ for all $i,j$.
\end{lemma}
\begin{proof}
If $w_i-w_j > 1$ for some pair $(w_i, w_j)$, by Lemma \ref{lem:max2} we can reduce $w_i$ and increase $w_j$ by $1$ repeatedly  until $w_i - w_j \leq 1$ for all $i,j$ without increasing $w_{\max}(n,k) = \prod_{i=1}^n \sum_{j=1}^k a_{\sigma_j(i)} = \prod_{i=1}^{n} k(a_1-d)+w_id$. 
If $n$ is even and $k$ is odd, $\sum_i w_i$ is not divisible by $n$ and the only set of $w_i$ such that $w_i - w_j \leq 1$ for all $i,j$ is the one described in Lemma \ref{lem:n-1seq}.
If $n$ is odd or $k$ is even, there exists a set of permutations corresponding to $w_{\max}(n,k)$ such that $w_i = w_j$ by Theorem \ref{thm:k=even}.
\end{proof}

\begin{theorem}\label{thm:k=odd}
If $n$ is even and $k$ is odd such that $k \geq n-1$, then 
\[ w_{\max}(n,k) = \left(ka_1 + \left(\frac{k(n-1)-1}{2}\right)d\right)^{n/2}  \left(ka_1 + \left(\frac{k(n-1)+1}{2}\right)d\right)^{n/2} \]
\end{theorem}
\begin{proof}
Note that $k$ can be written as $k = 2t+(n-1)$.
As a consequence of Lemmas \ref{lem:n-1seq}, \ref{lem:maxperm}, the value $w_{\max}(n,k)$ is achieved with $t$ copies of $(1,...,n)$, $t$ copies of $(n,...,1)$, $\tilde{\sigma}$ and the cyclic permutations with index in $S$.
Then $w_i = t(n+1) + n^2/2-1 = \frac{k(n+1)-1}{2}$ for $i=1,\cdots, n/2$, and $w_i = t(n+1)+n^2/2 = \frac{k(n+1)+1}{2}$ for $i = n/2+1,\cdots, n$.
Thus 
\begin{multline*}w_{\max}(n,k) = \prod_{i=1}^n k(a_1-d)+w_id \\ = \left(k(a_1-d)+ \frac{d(k(n+1)-1)}{2}\right)^{n/2} \left(k(a_1-d)+ \frac{d(k(n+1)+1)}{2}\right)^{n/2}
\end{multline*}
and the conclusion follows.
\end{proof}

Theorems \ref{thm:k=even} and \ref{thm:k=odd} show that the value of $w_{\max}(n,k)$ and the corresponding maximizing set of permutations can be explicitly found when $k\geq n-1$ or $k$ is  even.

\subsection{The special case {$a_i = i$}}
Consider the special case where the sequence $a_i$ is just the first $n$ positive integers, i.e. we have $v(n,k) = \sum_{i=1}^n \prod_{j=1}^k \sigma_j(i)$ and
$w(n,k) = \prod_{i=1}^n \sum_{j=1}^k \sigma_j(i)$. The values of $v_{\min}(n,k)$ and $w_{\max}(n,k)$ can be found in OEIS \cite{oeis} sequence A260355 (\url{https://oeis.org/A260355})
and sequence A331988 (\url{https://oeis.org/A331988}) respectively.

\begin{theorem}\label{thm:special_case}
If $k = 2t+nu$ for nonnegative integers $t$ and $u$, then $w_{\max}(n,k) = \left(\frac{k(n+1)}{2}\right)^n$. In particular, if $k$ is even or if  $n$ is odd and $k\geq n-1$, then  $w_{\max}(n,k) = \left(\frac{k(n+1)}{2}\right)^n$.
\end{theorem}

\begin{theorem}
If $n$ is even and $k$ is odd such that  $k \geq n-1$, then
$w_{\max}(n,k) = \left(\frac{k^2(n+1)^2 -1}{4}\right)^{n/2}$.
\end{theorem}

For example, Theorem \ref{thm:special_case} shows that $w_{\max}(3,k) = 8k^3$ for $k>1$. More details about $v_{\min}$ and $w_{\max}$ for this special case, including tables of values, can be found in Ref. \cite{wu:rearrange:arxiv2015}.

\section{The special case when $a_i$ is a geometric progression}
We can get analogous results for $v_{\min}$ if the sequence $a_i$ is a geometric progression of the form $a_i = cd^{b_i}$ for some constants $c, d\geq 1$ and an arithmetic progression $b_i$ of $n$ nonnegative numbers.
This is due to the fact that $\alpha_i \stackrel{\text{def}}{=} \log(a_i) = \log(c)+\log(d)b_i$ is an arithmetric progression of nonnegative numbers. Furthermore, if there exists permutations $\sigma_i$ such that
$\sum_i \alpha_{i\sigma_i(j)} = \sum_i \alpha_{i\sigma_i(1)}$, then
$\prod_i a_{i\sigma_i(j)} = \prod_i a_{i\sigma_i(1)}$.
This implies that we get the following analogous result to Theorem \ref{thm:k=even}.

\begin{theorem} \label{thm:geometric_k=even}
If $k = 2t+nu$ for nonnegative integers $t$ and $u$, then $v_{\min}(n,k) = n\prod_{i=1}^n a_i^{k/n}
= nc^kd^{\frac{k(b_1+b_n)}{2}}$.
\end{theorem}

\section{A generalization of the rearrangement inequality} \label{sec:generalizedRI}
In Ref. \cite{Day1972}, Eqs (\ref{eqn:rearrange1}-\ref{eqn:rearrange2}) are generalized as follows:

\begin{theorem} \label{thm:day1972}
Let $f$ be real valued function of 2 variables defined on $I_a\times I_b$.  If
$$f(x_2,y_2) - f(x_2,y_1) - f(x_1,y_2) + f(x_1,y_1) \geq 0$$ for all $x_1\leq x_2$ in $I_a$ and $y_1\leq y_2$ in $I_b$,
then 
\begin{equation} \label{eqn:rearrange-generalized}
\sum_i f(a_i, b_{n-i+1}) \leq \sum_i f(a_i, b_{\sigma(i)}) \leq \sum_{i} f(a_i, b_i)
\end{equation}
for all sequences $a_1\leq a_2\cdots \leq a_n$ in $I_a$,   $b_1\leq b_2\cdots \leq b_n$ in $I_b$, and all permutation $\sigma$ of $\{1,\cdots, n\}$ .
\end{theorem}

Theorem \ref{thm:day1972} unifies Eq. (\ref{eqn:rearrange1}) and Eq. (\ref{eqn:rearrange2}) as they can be derived by choosing $f(x,y) = xy$ and $f(x,y) = -\log(x+y)$ respectively. The assumption $x_i\geq 0$ and $y_i \geq 0$ 
in Eq. (\ref{eqn:rearrange2}) are used to ensure that the log is well-defined.
In this section, we generalize this theorem by replacing the summation and subtraction with a general function and real intervals with partially ordered sets and give a more direct way to unify Eq. (\ref{eqn:rearrange1}) and Eq. (\ref{eqn:rearrange2}).

\begin{definition}
For a function $g$ with $n$ arguments and for $i\neq j$ define $g_{ij}(x,y,z)$ as $g(z)$ but with the $i$-th and $j$-th argument replaced with $x$ and $y$ respectively. Similarly, we define $g_i(x,z)$ as $g(z)$ except with the $i$-th argument replaced with $x$.
\end{definition}
For instance if $g(z_1,z_2,z_3)$ is a function of 3 arguments, then $g_{1,3}(x,y,(z_1,z_2,z_3)) = g(x,z_2,y)$ and $g_2(x,(z_1,z_2,z_3)) = g(z_1,x,z_3)$.

\begin{definition}
A function $g$ on $n$ variables satisfies property $S$ if the value
$g(x_{\sigma(1)}, \cdots , x_{\sigma(n)})$ does not depend on the permutation $\sigma\in S_n$.
\end{definition}

\begin{theorem} \label{thm:newRI}
Let $I_a$ and $I_b$ be two sets with corresponding partial orders $\preceq_a$ and $\preceq_b$.
Let $f: I_a\times I_b \rightarrow I_c$ be a function of 2 variables defined on $I_a\times I_b$.  Let $g:I_c^n\rightarrow I_d$ be a function of n variables defined on $I_c^n$. Let $\preceq_d$ be a partial order on $I_d$.

If
\begin{equation}
g_{ij}(f(x_1,y_1), f(x_2,y_2),z) \succeq_d g_{ij}(f(x_2,y_1), f(x_1,y_2),z) \label{eqn:order1}
\end{equation}
 for all $x_1\preceq_a x_2$ in $I_a$ and $y_1 \preceq_b y_2$ in $I_b$ and all pairs of indices $i < j$ and all $z$,
then
\begin{equation}g(\{f(a_{i}, b_{n-i+1})|i = 1, ...,n\}) \\ \preceq_d g(\{f(a_i, b_{\sigma(i)})|i = 1,...,n\}) \preceq_d g(\{f(a_{i}, b_{i})|i = 1,...,n\}) \label{eqn:newRI} \end{equation}
for all sequences $a_1\preceq_a a_2\cdots \preceq_a a_n$ in $I_a$,   $b_1\preceq_b b_2\cdots \preceq_b b_n$ in $I_b$, and all permutation $\sigma \in S_n$.
\end{theorem}

\begin{proof}
The proof is similar to Ref. \cite{Vince1990} where we use the permutahedron ordering $P_n$ on $S_n$, with $\sigma_1 \succeq \sigma_2$ if $\sigma_1$ can be formed from $\sigma_2$ by exchanging the elements of an adjacent inversion and we consider the partial order on $S_n$ generated by the transitive closure of $P_n$.
 Let $x_1\preceq_a x_2 \preceq_a \cdots \preceq_a x_n$ and $y_1\preceq_b y_2 \preceq_b \cdots \preceq_b y_n$ and define $g^{\sigma} = g(f(x_1,y_{\sigma(1)}),f(x_2,y_{\sigma(2)}), ...)$ for $\sigma \in S_n$.
If $\sigma_1\succeq \sigma_2$, then Eq. (\ref{eqn:order1}) implies that
$g^{\sigma_1}\succeq_d g^{\sigma_2} $. Since the greatest element and the least element in the partial order is the identity and the reverse permutation respectively, the conclusion follows.
\end{proof}

A slight variation of Theorem \ref{thm:newRI} is the following:
\begin{theorem} \label{thm:newRI2}
Let $I_a$, $I_b$, $I_c$, $I_d$, $f$ and $g$ be as defined in Theorem \ref{thm:newRI}..

If Eq. (\ref{eqn:order1}) is satisfied
 for all $x_1\preceq_a x_2$ in $I_a$ and $y_1\preceq_b y_2$ in $I_b$ and all pairs of indices $i\neq j$ and all $z$, and $g$ satisfies property $S$, then
\begin{equation}g(\{f(a_{\mu(i)}, b_{\mu(n-i+1)})|i = 1, ...,n\}) \preceq_d g(\{f(a_i, b_{\sigma(i)})|i = 1,...,n\}) \\ \preceq_d g(\{f(a_{\mu(i)}, b_{\mu(i)})|i = 1,...,n\}) \label{eqn:newRI2} \end{equation}
for all sequences $a_1\preceq_a a_2 \preceq_a \cdots \preceq_a a_n$ in $I_a$,   $b_1\preceq_b b_2 \preceq_b\cdots \preceq_b b_n$ in $I_b$, and all permutation $\sigma, \mu \in S_n$.
\end{theorem}

The proof of Theorem \ref{thm:newRI2} is similar to Theorem \ref{thm:newRI} except that we define $g^{\sigma}$ as $$g^{\sigma} = g(f(x_{\mu(1)},y_{\mu(\sigma(1))}),f(x_{\mu(2)},y_{\mu(\sigma(2))}), ...).$$

\begin{lemma} \label{lem:prodsum}
Let $x_1$, $x_2$, $y_1$ and $y_2$ be real numbers.
If $x_1 \leq x_2$ and $y_1\leq y_2$, then
$$ x_1y_1 +x_2y_2 \geq x_1y_2+x_2y_1 $$
and
$$ (x_1+y_1)(x_2+y_2) \leq (x_1+y_2)(x_2+y_1) $$
\end{lemma}
\begin{proof}
The inequalities follow from the fact that they can both be rearranged into
$(x_2-x_1)(y_2-y_1) \geq 0$.
\end{proof}

Theorem \ref{thm:newRI} gives us a more direct way to unify Eq. (\ref{eqn:rearrange1}) and Eq. (\ref{eqn:rearrange2}). If we choose $g(x_1,x_2,...) = \sum_i x_i$ and $f(x,y) = xy$, then  Lemma \ref{lem:prodsum} implies that Eq. (\ref{eqn:order1}) is satisfied and we obtain Eq. (\ref{eqn:rearrange1}).
If we choose $g(x_1,x_2,...) = -\prod_i x_i$ and $f(x,y) = x+y$, then  Lemma \ref{lem:prodsum} with the additional assumption that $x_i, y_i \geq 0$ ensures that $z \geq 0$ and thus Eq. (\ref{eqn:order1}) is satisfied and we obtain Eq. (\ref{eqn:rearrange2}). Not having to use the log function to prove Eq. (\ref{eqn:rearrange2}) will be useful when we look at more general products such as the Hadamard product of matrices in Section \ref{sec:RIcalc}.

Other choices of $f$ and $g$ beyond addition and multiplication are for example $\max$ and $\min$ functions. Table \ref{tabl:fg} lists some of these choices for $f$ and $g$ that satisfies Eq. (\ref{eqn:order1}) where $\mathbb{R}_{\geq 0}$ denotes the set of nonnegative real numbers. The case of $g=\prod$ and $f=\min$ was also shown in Ref. \cite{jurkat:1967}.

\begin{table}[htbp]
\begin{center}
\begin{tabular}{|c|c|c|}
\hline
$f(x_1,x_2)$ & $g(x_1,\cdots, x_n)$ & Domain \\
\hline\hline
\multirow{3}{*}{$x_1\times x_2$} & $\sum_i x_i$ & $\mathbb{R}$ \\ \cline{2-3}
			& $\max_i x_i$ & $\mathbb{R}_{\geq 0}$ \\ \cline{2-3}
			& $-\min_i x_i$ & $\mathbb{R}_{\geq 0}$
 \\\hline
\multirow{3}{*}{$x_1+x_2$} & $-\Pi_i x_i$ & $\mathbb{R}_{\geq 0}$ \\ \cline{2-3}
			& $\max_i x_i$ & $\mathbb{R}$ \\ \cline{2-3}
			& $-\min_i x_i$ & $\mathbb{R}$ \\ 
 \hline
\multirow{3}{*}{$\max(x_1,x_2)$} & $-\sum_i x_i$ & $\mathbb{R}$ \\ \cline{2-3}
                       & $-\Pi_i x_i$ & $\mathbb{R}_{\geq 0}$ \\ \cline{2-3}
                       & $-\min_i x_i$ & $\mathbb{R}$ 
\\ \hline
\multirow{3}{*}{$\min(x_1,x_2)$} & $\sum_i x_i$ & $\mathbb{R}$ \\ \cline{2-3}
			& $\Pi_i x_i$ & $\mathbb{R}_{\geq 0}$ \\  \cline{2-3}
			 & $\max_i x_i$ & $\mathbb{R}$ \\ \hline
\end{tabular}
\end{center}
\caption{Examples of functions $f$ and $g$ such that Eq. (\ref{eqn:order1}) is satisfied. All the functions $g$ satisfy property $S$.}
\label{tabl:fg}
\end{table}

We will look at more general examples in Section \ref{sec:RIcalc}.

\subsection{Circular rearrangement inequality} \label{sec:circular}
In Ref. \cite{Yu2018} the following variant of the rearrangement inequality is studied for a sequence of numbers $a_1\leq a_2\leq \cdots \leq a_n$. Consider the value
$V(\sigma) = a_{\sigma(1)}a_{\sigma(2)} + a_{\sigma(2)}a_{\sigma(3)}  + \cdots + a_{\sigma(n)}a_{\sigma(1)} $, where $\sigma$ is a permutation of $\{1,2,\cdots n\}$.
Let $\sigma_{m_1}$ denote the permutation $(1,n-1,3,n-3,5,\cdots, n-6,6,n-4,4,n-2,2,n)$ and
$\sigma_{m_2}$ denote the permutation $(1,3,5,\cdots, n, \cdots, 6,4,2)$.
It was shown in Ref. \cite{Yu2018} that $V(\sigma)$ is minimized and maximized when the permutation $\sigma$ is equal to $\sigma_{m_1}$ and $\sigma_{m_2}$ respectively.

As the proof of this result only relies on properties of addition and multiplication described in Lemma \ref{lem:prodsum}, the following extension follows readily:
\begin{theorem} \label{thm:circular}
If $f$, $g$ satisfies Eq. (\ref{eqn:order1}), $f$ is symmetric, $g$ satisfies property $S$ and $a_1\leq a_2\leq \cdots \leq a_n$, then the value of 
$$ g\left(f\left(a_{\sigma(1)},a_{\sigma(2)}\right), f\left(a_{\sigma(2)},a_{\sigma(3)}\right), \cdots,  f\left(a_{\sigma(n)},a_{\sigma(1)}\right)\right)$$ 
is minimized and maximized when the permutation $\sigma$ is equal to $\sigma_{m_1}$ and $\sigma_{m_2}$ respectively.
\end{theorem}

A consequence is the dual to the result in Ref. \cite{Yu2018}.
\begin{corollary}
If $0\leq a_1\leq a_2\leq \cdots \leq a_n$, then the value of 
$$W(\sigma) = \left(a_{\sigma(1)}+a_{\sigma(2)}\right)\left(a_{\sigma(2)}+a_{\sigma(3)}\right)\times  \cdots \times \left(a_{\sigma(n)}+a_{\sigma(1)}\right)$$
is minimized and maximized when the permutation $\sigma$ is equal to $\sigma_{m_2}$ and $\sigma_{m_1}$ respectively.
\end{corollary}

\subsection{Extension to multiple sequences}

Theorem \ref{thm:newRI}  can be generalized to multiple sequences as well.

\begin{theorem} \label{thm:newRImulti}
Let $f$ be a function of $k$ variables and let $g$ be function of $n$ variables.

If
\begin{equation}
g_{ij}(f_{ml}(x_1,y_1,w), f_{ml}(x_2,y_2,w),z) \succeq_d g_{ij}(f_{ml}(x_2,y_1,w), f_{ml}(x_1,y_2,w),z) \label{eqn:order2multi}
\end{equation}
 for all $x_1\preceq_i x_2$ and $y_1 \preceq_j y_2$ and all pairs of indices $i< j$, $m< l$ and all $z$, $w$, then
$$g(\{f(a_{1i}, a_{2\sigma_2(i)},\cdots, a_{k\sigma_k(i)}|i = 1,...,n\}) \preceq_d g(\{f(a_{1i},a_{2i},\cdots a_{ki})|i = 1,\cdots,n\})$$ 
for all permutations $\sigma_j\in S_n$ and for all sequences $a_{ij}$, $1\leq i\leq k$, $1\leq j\leq n$ where for all $i$, $a_{i1}\preceq_i a_{i2}\preceq_i \cdots \preceq_i a_{in}$.
\end{theorem}

\begin{proof}
This follows by induction on the number of arguments of $f$ and the fact that once all the sequences are similarly ordered, exchanging any pair of adjacent terms in one sequence will not increase the value of $g$ as a consequence of Eq. (\ref{eqn:order2multi}).
\end{proof}

The corresponding extension of Theorem \ref{thm:newRI2}  to multiple sequences is
\begin{theorem} \label{thm:newRImulti2}
Let $f$ be a function of $k$ variables and let $g$ satisfies property $S$.
If Eq. (\ref{eqn:order2multi}) is satisfied
 for all $x_1\preceq_i x_2$ and $y_1\preceq_j y_2$ and all pairs of indices $i\neq j$, $m< l$ and all $z$, $w$, then
\begin{equation*}g(\{f(a_{1\sigma_1(i)}, a_{2\sigma_2(i)},\cdots, a_{k\sigma_k(i)}|i = 1,...,n\}) \\ \preceq_d g(\{f(a_{1\mu(i)},a_{2\mu(i)},\cdots, a_{k\mu(i)})|i = 1,\cdots,n\})
\end{equation*}
for all permutations  $\mu$, $\sigma_j\in S_n$ and for all sequences $a_{ij}$, $1\leq i\leq k$, $1\leq j\leq n$ where for all $i$, $a_{i1}\preceq_i a_{i2}\preceq_i \cdots \preceq_i a_{in}$.
\end{theorem}

Similarly, if the functions $f$ in Table \ref{tabl:fg} are extended as functions of $k$ variables and the domain is restricted to $\mathbb{R}_{\geq 0}$ they would satisfy Eq. (\ref{eqn:order2multi}).

\section{Another variation of the rearrangement inequality}
In Theorem \ref{thm:newRI}, the sequences $a_i$ and $b_i$ are separate and the permutation $\sigma$ acts on $b_i$ only. We next introduce a variant of the rearrangement inequality where the permutation acts on the union of $a_i$ and $b_i$.

\begin{theorem} \label{thm:variation}
Let $I$ be a set with partial order $\preceq$ and let $f:I\times I\rightarrow I_c$ be a function of 2 variables. Let $g:I_c^n\rightarrow I_d$ be a function of $n$ variables. Let $\preceq_c$ and $\preceq_d$ be partial orders for sets $I_c$ and $I_d$ respectively. 
Let $a_i$ be a set of $2n$ elements in $I$ such that $a_1\preceq a_2\preceq \cdots \preceq a_{2n}$ and  let $b_i$ be any permutation of the elements of $a_i$.
If $x\preceq_c y\Rightarrow g_i(x)\preceq_d g_i(y)$ for all $i$ and
\begin{equation}\label{eqn:asymmetry} 
f(x_1,x_2) \preceq_c f(x_2,x_1),
\end{equation}
and
\begin{equation}
g_{ij}(f(x_1,y_1), f(x_2,y_2),z) \succeq_d g_{ij}(f(x_2,y_1), f(x_1,y_2),z) \label{eqn:order1variation}
\end{equation}
 for all $x_1\preceq x_2$ and $y_1\preceq y_2$ in $I$ and all pairs of indices $i < j$ and all $z$, then 
\begin{equation}\label{eqn:rearrange-general1}
g(\{f(a_{i},a_{2n-i+1})|i = 1,\cdots n\}) \preceq_d g(\{f(b_{2i-1},b_{2i})|i = 1,\cdots n\})
\end{equation}

If $f$ is symmetric, i.e.
\begin{equation}\label{eqn:symmetry} 
f(x, y) = f(y,x)
\end{equation}
for all $x$, $y$ in $I$, and Eq. (\ref{eqn:order1variation}) is satisfied  for all $x_1\preceq x_2$ and $y_1\preceq y_2$ in $I$ and all pairs of indices $i < j$ and all $z$, then 
\begin{equation}\label{eqn:rearrange-general2}  
g(\{f(b_{2i-1},b_{2i})|i = 1,\cdots n\}) \preceq_d g(\{f(a_{2i-1},a_{2i})|i = 1,\cdots n\})
\end{equation}

\end{theorem}

\begin{proof}
Let $c_i$ be a permutation of $b_i$ such that $v = g(\{f(c_{i},c_{2n-i+1}|i = 1,\cdots, n\})$ is a minimal element under $\preceq_d$. Then 
by Theorem \ref{thm:newRI}, $c_i$ can be chosen such that $c_i \preceq c_{i+1}$ for $1\leq i\leq n-1$ and for $n+1\leq i \leq 2n-1$.
Suppose $c_{n+1} \prec c_{n}$. By Eq. (\ref{eqn:asymmetry}) we can swap these two terms without causing $v$ to be nonminimal.
Again by Theorem \ref{thm:newRI}, we can reorder $c_i$ for $1\leq i\leq n$ such that they are nondecreasing under $\preceq$ and also reorder $c_i$ for $n+1\leq i\leq 2n$ such that they are nondecreasing. 
If $c_{n+1} \prec c_{n}$ we repeat the process again. It's clear that this needs to be repeated at most a finite number of times and eventually we have $c_{n+1} \succeq c_{n}$.
Thus we have a sequence of $c_i$ such that $c_i \preceq c_{i+1}$ for $1\leq i\leq n-1$ and for $n+1\leq i \leq 2n-1$, in addition to $c_{n} \preceq c_{n+1}$, i.e., $c_1\preceq c_2 \cdots \preceq c_{2n}$. Since each swap of 2 elements in the permutation results in comparable elements in $I_d$, this minimal element $v$ is also the least element $v$ under $\preceq_d$ among all the permutations of $b_i$.

Next, let $d_i$ be a permutation of $b_i$ such that $v = g(\{ f(d_{2i-1},d_{2i}\}|n = 1,\cdots, n\})$ is a maximal element under $\preceq_d$. Then 
by Theorem \ref{thm:newRI}, $d_i$ can be chosen such that $d_{2i-1} \preceq d_{2i+1}$ and $d_{2i} \preceq d_{2i+2}$ for $1\leq i\leq n-1$. 
Furthermore, by repeated use of  Theorem \ref{thm:newRI} and Eq. (\ref{eqn:symmetry}) we can assume $d_{2i-1} \preceq d_{2i}$ as well.
Suppose $d_{2n-1} \prec d_{2(n-1)}$. Then $d_{2(n-1)-1} \prec d_{2(n-1)}$ and by Eq. (\ref{eqn:symmetry}) we can swap $d_{2(n-1)}$ and $d_{2(n-1)-1}$ without changing the value of $v$.
Again by repeated application of Theorem \ref{thm:newRI} and Eq. (\ref{eqn:symmetry}) we can reorder $d_{2i}$ for $1\leq i\leq n$ such that they are nondecreasing under $\preceq$ and also reorder $d_{2i-1}$ for $1\leq i\leq n$ such that they are nondecreasing in addition to ensuring $d_{2i-1} \preceq d_{2i}$ without changing $v$. It is easy to see that after this reordering $d_{2n-1} \succeq d_{2(n-1)}$. Applying this procedure for $j = n-1, ..., 3, 2$ sequentially shows that for each $2\leq j\leq n$, $d_{2j-1} \succeq d_{2(j-1)}$.  This in addition with the fact that $d_{2i}\succeq d_{2i-1}$ shows that $d_1\preceq d_2 \cdots \preceq d_{2n}$. Similarly, this maximal element $v$ is also the greatest element $v$ among all the permutations of $b_i$.
\end{proof}

By choosing $g(x_1,x_2,\cdots) = \sum_i x_i $ and $f(x,y) = xy$ or  $g(x_1, x_2, \cdots) = - \prod_i x_i$ and $f(x,y) = x+y$, we have the following result.

\begin{corollary} \label{cor:variation-sumprod}
Let $a_i$ be a set of $2n$ numbers and let $b_i$ be the numbers $a_i$ sorted such that $b_1 \leq b_2\leq \cdots \leq b_{2n}$. 
Then 
$$\sum_{i=1}^n b_{i}b_{2n-i+1} \leq   \sum_{i=1}^n a_{2i-1}a_{2i} \leq \sum_{i=1}^n b_{2i-1}b_{2i}.$$
If in addition $a_i\geq 0$, then
$$\prod_{i=1}^n \left(b_{2i-1}+b_{2i}\right) \leq   \prod_{i=1}^n \left(a_{2i-1}+a_{2i}\right) \leq \prod_{i=1}^n \left(b_{i}+b_{2n-i+1}\right).$$
\end{corollary}

It is interesting to note that when $\{a_i\} = \{x_1, x_1, x_2, x_2, \cdots, x_n, x_n\}$ consists of $n$ numbers each occuring twice, then the optimal permutations in Corollary \ref{cor:variation-sumprod} correspond to the optimal permutations in Eqns. (\ref{eqn:rearrange1}-\ref{eqn:rearrange2}).

Similarly, we can generalize Theorem \ref{thm:newRImulti} to multiple sequences when the permutation is among all $kn$ numbers $\{a_{ij}\}$.

\begin{theorem}\label{thm:multiple-seq-variant}
Consider a sequence of $kn$ elements $a_{i}$ in $I$ with partial order $\preceq$ such that
$a_1\preceq a_2 \preceq \cdots \preceq a_{kn}$. Let
$\{b_{i}\}$ be and arbitrary permutation of $\{a_i\}$.  
Let $f(x_1,\cdots, x_k)$ be a function defined on $I^k$ such that 
$$f_{ml}(x,y,z) = f_{ml}(y,x,z)$$ for all $x$, $y$, $z$ and pairs of indices $m < l$ and
 Eq. (\ref{eqn:order1variation}) is satisfied  for all $x_1\preceq x_2$ and $y_1\preceq y_2$ in $I$ and all pairs of indices $i < j$ and all $z$, then
\begin{equation*} g(\{f(b_{(j-1)k+1},b_{(j-1)k+2},\cdots, b_{jk}|j=1.\cdots, n\}) \\ \preceq_d   g(\{f(a_{(j-1)k+1},a_{(j-1)k+2},\cdots, a_{jk}|j=1,\cdots, n\}) \end{equation*}
\end{theorem}
\begin{proof}
The proof is similar to Theorem \ref{thm:variation}.
Let $d_i$ be a permutation of $b_i$ such that $$v =  g(\{f(d_{(j-1)k+1},d_{(j-1)k+2},\cdots, d_{jk}|j=1,\cdots n\})$$ is a maximal element. Then 
by Theorem \ref{thm:newRImulti}, $d_i$ can be chosen such that $d_{(j-1)k+i} \preceq d_{jk+i}$ for $1\leq i\leq k$ and $1\leq j\leq n-1$. Furthermore, by Eq. (\ref{eqn:symmetry}) we can also assume that $d_{(j-1)k+i} \preceq  d_{(j-1)k+i+1}$ for $1\leq i\leq k-1$ and $1\leq j\leq n$.
Suppose $d_{k(n-1)+1} \prec d_{k(n-1)}$. By Eq. (\ref{eqn:symmetry}) we can swap $d_{k(n-2)+1}$ and $d_{k(n-1)}$ without changing the value of $v$.
Again by repeated application of Eq. (\ref{eqn:symmetry}) and Theorem \ref{thm:newRI}, we can reorder $d_{i}$  such that $d_{(j-1)k+i} \preceq d_{jk+i}$ for $1\leq i\leq k$ and $1\leq j\leq n-1$ without changing $v$ while ensuring $d_{(j-1)k+i} \preceq  d_{(j-1)k+i+1}$ for $1\leq i\leq k-1$ and $1\leq j\leq n$. If $d_{k(n-1)+1} \prec d_{k(n-1)}$ we repeat this process (which terminates after a finite number of times) until $d_{k(n-1)+1} \succeq d_{k(n-1)}$. Applying this procedure for $j$ from $n-1,\cdots, 3, 2$ sequentially shows that for each $2\leq j\leq n$, $d_{(j-1)k+1}
\succeq  d_{k(j-1)}$.  This along with  $d_{(j-1)k+i} \preceq  d_{(j-1)k+i+1}$ for $1\leq i\leq k-1$ and $1\leq j\leq n$ shows that $d_1\preceq d_2\preceq \cdots \preceq d_{kn}$.
\end{proof}

We get the following result when we set   $g(x_1,\cdots , x_n) = \sum_{i=1}^n x_i$ and $f(x_1,\cdots , x_k) = \prod_{i=1}^k x_i$ or if we set  $g(x_1,\cdots , x_n) = -\prod_{i=1}^n x_i$ and $f(x_1,\cdots , x_k) = \sum_{i=1}^k x_i$.
\begin{corollary} \label{cor:variation-prodsum-multiple}
Let $a_i\geq 0$ be a set of $kn$ numbers and let $b_i$ be the numbers $a_i$ reordered such that $b_1 \leq b_2 \leq \cdots \leq b_{kn}$. 
Then 
$$n\sqrt[n]{\prod_{i=1}^{kn} a_i}\leq \sum_{j=1}^n \prod_{i=1}^k a_{(j-1)k+i} \leq \sum_{j=1}^n \prod_{i=1}^k b_{(j-1)k+i}$$
and 
$$\prod_{j=1}^n \sum_{i=1}^k b_{(j-1)k+i} \leq \prod_{j=1}^n \sum_{i=1}^k a_{(j-1)k+i} \leq \left(\frac{\sum_{i=1}^{kn} a_i}{n}\right)^n.$$
Suppose there exists $c_i$, a reordering of the numbers $a_i$ such that  $\prod_{i=1}^k c_{(j-1)k+i} =  \prod_{i=1}^k c_{(l-1)k+i}$ for all $1\leq j,l\leq n$.
Then 
$$\sum_{j=1}^n \prod_{i=1}^k c_{(j-1)k+i} \leq \sum_{j=1}^n \prod_{i=1}^k a_{(j-1)k+i}$$
Suppose there exists $c_i$, a reordering of  the numbers $a_i$ such that  $\sum_{i=1}^k c_{(j-1)k+i} =  \sum_{i=1}^k c_{(l-1)k+i}$ for all $1\leq j,l\leq n$, then
$$\prod_{j=1}^n \sum_{i=1}^k a_{(j-1)k+i} \leq \prod_{j=1}^n \sum_{i=1}^k c_{(j-1)k+i}$$

\end{corollary}

The bounds $n\sqrt[n]{\prod_{i=1}^{kn} a_i}$ and $\left(\frac{\sum_{i=1}^{kn} a_i}{n}\right)^n$ in Corollary \ref{cor:variation-prodsum-multiple} are due to the AM-GM inequality (Lemma \ref{lem:am-gm}).

\subsection{The special case when $a_i$ is an arithmetic progression}
In general, Corollary \ref{cor:variation-prodsum-multiple} provides a tight bound only on one side. On the other hand, both a tight upper and lower bound can be derived under certain conditions when the numbers $a_i$ form an arithmetic progression.

\begin{definition}
For a permutation $\sigma$ of $\{1,\cdots, kn\}$, define 
$$v(n,k) = \sum_{i=1}^{n}\prod_{j=1}^k a_{\sigma((i-1)k+j)}.$$ Let $v_{\min}(n,k)$ and $v_{\max}(n,k)$ be the minimal and maximal values respectively of $v(n,k)$ among all permutations $\sigma$ of $\{1,\cdots, kn\}$.
\end{definition}

\begin{definition}
For a permutation $\sigma$ of $\{1,\cdots, kn\}$, define 
$$w(n,k) = \prod_{i=1}^{n}\sum_{j=1}^k a_{\sigma((i-1)k+j)}.$$  Let $w_{\min}(n,k)$ and $w_{\max}(n,k)$ be the minimal and maximal values respectively of $w(n,k)$ among all permutations $\sigma$ of $\{1,\cdots, kn\}$.
\end{definition}

Suppose $a_i\geq 0$ is an arithmetic progression, with $a_i = a_1 + (i-1)d$, for $i=1,\cdots, kn$, $d\geq 0$. Corollary \ref{cor:variation-prodsum-multiple} implies that
\begin{theorem}
\begin{itemize}
\item $v_{\min}(n,k) \geq nd^k\sqrt[n]{\frac{\Gamma\left(\frac{a_1}{d}+nk\right)}{\Gamma\left(\frac{a_1}{d}\right)}}$.
\item $v_{\max}(n,k) = \sum_{i=1}^{n}\prod_{j=1}^k a_{(i-1)k+j} = d^k\sum_{i=1}^{n} \frac{\Gamma\left(\frac{a_1}{d} + ik\right)}{\Gamma\left(\frac{a_1}{d}+(i-1)k\right)}$.
\item $w_{\max}(n,k) \leq \left(\frac{k\left(a_1+a_{kn}\right)}{2}\right)^n$.
\item \begin{equation*}w_{\min}(n,k) = \prod_{i=1}^{n}\sum_{j=1}^k a_{(i-1)k+j} \\ = k^n\prod_{i=1}^n \left(a_1+\left(ik-\frac{k+1}{2}\right)d\right) = k^{2n}d^{n}\frac{\Gamma\left(n+\frac{2a_1+(k-1)d}{2kd}\right)}{\Gamma\left(\frac{2a_1+(k-1)d}{2kd}\right)}.\end{equation*}
\end{itemize}
\end{theorem}

\begin{theorem}
If $k = 2t+nu$ for nonnegative integers $t$ and $u$, then 
$w_{\max}(n,k) = \left(\frac{k(a_1+a_{kn})}{2}\right)^n$.
\end{theorem}
\begin{proof}
The proof is similar to the proof of Theorem
 \ref{thm:k=even}. Instead of using cyclic permutations $r_i$ of $\{1,\cdots, n\}$ and the permutation $(n,n-1,\cdots, 1)$, we apply them to $((j-1)n+1,(j-1)n+2, \cdots, jn)$ and this is equivalent to adding $(j-1)n$ to each term of the $j$-th permutation. For instance, for $n=k=3$, $w(n,k)$ is maximized by $(a_1,a_5,a_9,a_2,a_6,a_7,a_3,a_4,a_8)$.
 \end{proof}
 
 This implies that if $n$ is odd and $k\geq n-1$ or if $k$ is even, then $w_{\max}(n,k) = \left(\frac{k(a_1+a_{kn})}{2}\right)^n$.
 
 \begin{theorem} \label{thm:k=odd-variation}
 If $n$ is even and $k$ is odd such that $k\geq n-1$, then
 \[ w_{\max}(n,k) = \left(ka_1 + \left(\frac{k(kn-1)-1}{2}\right)d\right)^{n/2}  \left(ka_1 + \left(\frac{k(kn-1)+1}{2}\right)d\right)^{n/2} \]
 \end{theorem}
\begin{proof} The proof is similar to the proof of Theorem \ref{thm:k=odd}, except that we add $(j-1)n$ to each term of the $j$-th permutation in the $k$-set of permutations of $\{1,\cdots , n\}$.
This adds an additional $\sum_{j=1}^k (j-1)n = (k-1)kn/2$ to each $w_i$ and thus $w_i = \frac{k(kn+1)-1}{2}$ for $i=1,\cdots, n/2$, and $w_i = \frac{k(kn+1)+1}{2}$ for $i = n/2+1,\cdots, n$.
Thus $w_{\max}(n,k) = \prod_{i=1}^n k(a_1-d)+w_id = \left(k(a_1-d)+ \frac{d(k(kn+1)-1)}{2}\right)^{n/2} \left(k(a_1-d)+ \frac{d(k(kn+1)+1)}{2}\right)^{n/2}$ and the conclusion follows.
\end{proof}

Analogous to Theorem \ref{thm:geometric_k=even}, we have the following result for a geometric progression:

\begin{theorem}
For a geometric progression sequence $a_i = cd^{b_i}$ where $c,d\geq 1$ and $b_i$ is an arithmetic progression of $kn$ nonnegative numbers, if $k=2t+nu$ for $t,u\geq 0$, then
$v_{\min}(n,k) = n\prod_{i=1}^{kn} a_i^{1/n} = nc^kd^{\frac{k(b_1+b_{kn})}{2}}$.
\end{theorem}

\subsection{The special case $a_i = i$}

\begin{definition}
For a permutation $\sigma$ of $\{1,\cdots, kn\}$, define $$v(n,k) = \sum_{i=1}^{n}\prod_{j=1}^k \sigma((i-1)k+j).$$ Let $v_{\min}(n,k)$ and $v_{\max}(n,k)$ be the minimal and maximal values respectively of $v(n,k)$ among all permutations $\sigma$ of $\{1,\cdots, kn\}$.
\end{definition}

\begin{definition}
For a permutation $\sigma$ of $\{1,\cdots, kn\}$, define 
$$w(n,k) = \prod_{i=1}^{n}\sum_{j=1}^k \sigma((i-1)k+j).$$ Let $w_{\min}(n,k)$ and $w_{\max}(n,k)$ be the minimal and maximal values respectively of $w(n,k)$ among all permutations $\sigma$ of $\{1,\cdots, kn\}$.
\end{definition}

We have $v_{\min}(n,1) = w_{\max}(1,n) = n(n+1)/2$, $v_{\min}(1,k) = w_{\max}(k,1) = k!$, and $v_{\min}(n,k) \geq n\sqrt[n]{(kn)!}$. Furthermore, $w_{\max}(n,k) \leq \left(\frac{k(nk+1)}{2}\right)^n$ with equality if $k = 2t+nu$ for nonnegative integers $t$ and $u$.
\begin{theorem} $v_{\min}(n,2) = n(n+1)(2n+1)/3$, $w_{\max}(n,2) = (2n+1)^n$. 
\end{theorem}

\begin{proof}
By Corollary \ref{cor:variation-prodsum-multiple}, $v_{\min}(n,2) = \sum_{i=1}^n i(2n-i+1) = (2n+1)\sum_{i}^n i - \sum_{i}^n i^2 = n(n+1)(2n+1)/2 - n(n+1)(2n+1)/6 =  n(n+1)(2n+1)/3$. Similarly,
$w_{\max}(n,2) = \prod_{i=1}^n (i+(2n-i+1)) = (2n+1)^n$. 
\end{proof}

Theorem \ref{thm:k=odd-variation} implies that
\begin{corollary}
If $n$ is even and $k$ is odd such that $k\geq n-1$, then $w_{\max}(n,k) = \left(\frac{k^2(kn+1)^2 -1}{4}\right)^{n/2}$.
\end{corollary}

The value of $v_{\min}(n,3)$ can be found in  OEIS \cite{oeis} as OEIS sequence A072368 (\url{https://oeis.org/A072368}). The values of $v_{\min}(n,k)$ can be found in 
sequence A331889 (\url{https://oeis.org/A331889}). The values of $w_{\max}(n,k)$ can be found in 
sequence A333420 (\url{https://oeis.org/A333420}).  The values of $w_{\min}(n,k)$ can be found in 
sequence A333445 (\url{https://oeis.org/A333445}). The values of $v_{\max}(n,k)$ can be found in 
sequence A333446 (\url{https://oeis.org/A333446}).

\section{Rearrangement inequalities for generalized sum-of-products and product-of-sums} \label{sec:RIcalc}
So far the examples above deal mainly with sequences of real numbers. In this section we look at other partially ordered sets for which Eq. (\ref{eqn:order1}) can be satisfied.

\begin{definition}[Ref. \cite{mathdict:1993}]
A partially ordered group $(G,+,\preceq)$ is defined as a group $G$ with group operation $+$ and a partial order $\preceq$ on $G$ such that 
 $z+x\preceq z+y\Leftrightarrow x+z \preceq y+z \Leftrightarrow x\preceq y$ for all $x,y,z\in G$. 
\end{definition}

\begin{definition} \label{def:cg}
Define $\cal C$ as the set of tuples $\ijktuple$ satisfying the following conditions:
\begin{enumerate}
\item $(I,+_I,\preceq_I)$,  $(J,+_J,\preceq_J)$ and  $(K,+_K,\preceq_K)$ are partially ordered Abelian groups.
\item $\ast : I\times J\rightarrow K$ is a {\em distributive} operation, i.e. it satisfies $(x+_Iy)\ast z = x\ast z +_K y\ast z$ and $x\ast (y+_J z) = x\ast y +_K x\ast z$.
\item $\ast$ is nonnegativity-preserving: if $x\succeq_I 0$ and $y\succeq_J 0$, then $x\ast y\succeq_K 0$.
\end{enumerate}
\end{definition}

If $(I,+,\preceq_I))$ is a partially ordered group with an associative, distributive and nonnegativity preserving operation $\ast : I\times I \rightarrow I$ whose identity is in $I$, then $\ituple$ is a partially ordered ring. If in addition $\ast$ is commutative, then $\ituple$ is a partially ordered commutative ring.
We will write a tuple in $\cal C$ as $(I, \preceq_I, J, \preceq_J, K, \preceq_K, \ast)$
if the group operations $+_I$, $+_J$ and $+_K$ can be deduced from context.

\begin{landscape}
\begin{table}[htbp]
\begin{center}
\begin{tabular}{|c|c|c|c|c|c|c|c|}
\hline
$I$& $\preceq_I$& $J$& $\preceq_J$& $K$& $\preceq_K$& $\ast$&\shortstack{symmetric \\$\ast$}\\
\hline\hline
$\mathbb{R}$ & $\leq$ & $\mathbb{R}$  &  $\leq$   & $\mathbb{R}$ & $\leq$  & multiplication & yes \\
\hline
$\mathbb{R}^n$ & \shortstack{induced by\\positive\\cone} & $\mathbb{R}^n$  &  \shortstack{induced by\\positive\\cone}   & $\mathbb{R}$ & $\leq$ & dot product & yes \\
\hline
$\mathbb{R}^n$ &  \shortstack{induced by\\positive\\cone}  & $\mathbb{R}^n$  & \shortstack{induced by\\positive\\cone}   & $\mathbb{R}$ & $\leq$ & \shortstack{$x\ast y = x^TAy$ \\with $A > 0$} & \shortstack{no\\yes if $A=A^T$} \\
\hline
$f:[0,1]\rightarrow\mathbb{R}$ & \shortstack{induced by\\positive\\cone} & $f:[0,1]\rightarrow\mathbb{R}$  & \shortstack{induced by\\positive\\cone}   & $\mathbb{R}$ & $\leq$ & \shortstack{$f\ast g =$\\$\int_{0}^1 f(x)g(x) dx$} & yes \\
\hline
$\mathbb{R}^n$ & \shortstack{induced by\\positive\\cone} & $\mathbb{R}^n$  & \shortstack{induced by\\positive\\cone}   & $\mathbb{R}^n$ &  \shortstack{induced by\\positive\\cone}  & \shortstack{Hadamard\\ product}  & yes \\
\hline
$\mathbb{R}^{n\times n}$ &  \shortstack{induced by\\positive\\cone}  & $\mathbb{R}^{n\times n}$  & \shortstack{induced by\\positive\\cone}   & $\mathbb{R}^{n\times n}$ &  \shortstack{induced by\\positive\\cone} & \shortstack{Matrix \\multiplication} & no \\
\hline
\shortstack{Hermitian\\ matrices} &  \shortstack{ Loewner\\ order}  & \shortstack{Hermitian\\ matrices}  & \shortstack{ Loewner\\ order}  & $\mathbb{R}$ & $\leq$ & \shortstack{Frobenius\\ inner\\ product} & yes \\
\hline
\shortstack{Commuting \\Hermitian\\ matrices}  & \shortstack{ Loewner\\ order}  &\shortstack{Commuting \\Hermitian\\ matrices} &   \shortstack{ Loewner\\ order}   & \shortstack{Hermitian\\ matrices}  &  \shortstack{ Loewner\\ order}  & \shortstack{Matrix \\multiplication} & yes \\
\hline
\shortstack{Hermitian\\ matrices} &  \shortstack{ Loewner\\ order}  & \shortstack{Hermitian\\ matrices}  & \shortstack{ Loewner\\ order}  &  \shortstack{Hermitian\\ matrices}& \shortstack{ Loewner\\ order} & \shortstack{Hadamard\\ product} & yes \\
\hline
\shortstack{Hermitian\\ matrices} &  \shortstack{ Loewner\\ order}  & \shortstack{Hermitian\\ matrices}  & \shortstack{ Loewner\\ order}  &  \shortstack{Hermitian\\ matrices}& \shortstack{ Loewner\\ order} & \shortstack{Kronecker\\ product} & no \\
\hline
\shortstack{Hermitian\\ matrices} &  \shortstack{ Loewner\\ order}  & \shortstack{Hermitian\\ matrices}  & \shortstack{ Loewner\\ order}  &  \shortstack{Hermitian\\ matrices}& \shortstack{ Loewner\\ order} & \shortstack{reverse Kronecker\\ product\tablefootnote{The reverse Kronecker product $A \otimes_r B$ is defined as $B\otimes A$.}} & no \\
\hline
\end{tabular}
\end{center}
\caption{Examples of members in $\cal C$.}\label{tbl:calc}
\end{table}
\end{landscape}

Structures in $\cal C$ have been useful in extending Schur's inequality \cite{wu:schur:2021}. Examples of elements in $\cal C$ are listed in Table \ref{tbl:calc}.
Analogous to Lemma \ref{lem:prodsum} we have
\begin{lemma}\label{lem:calc}
Let $a_1, a_2\in I$, $b_1,b_2\in J$.
If  $a_1 \preceq_I a_2$ and $b_1 \preceq_J b_2$, then
$$  (a_1 \ast b_1) +_K  (a_2\ast b_2) \succeq_K  (a_1\ast b_2) +_K (a_2\ast b_1)$$
If addition $I = J$ and $\ast$ is symmetric, then
$$ ( a_1+_Ib_1)\ast ( a_2+_Ib_2)\preceq_K ( a_1+_Ib_2)\ast ( a_2+_Ib_1)$$
\end{lemma}

\begin{proof}
This follows from the fact that both inequalities can be rewritten as $(a_2-a_1)\ast (b_2-b_1) \succeq_K 0$.
\end{proof}

\begin{lemma}\label{lem:calc_multi}
Let $I=J=K$ and $a_1, a_2, b_1,b_2, c\in I$.
If  $0 \preceq_I a_1 \preceq_I a_2$, $0 \preceq_I b_1 \preceq_I b_2$, and $c \succeq_I 0$ then
\begin{equation}  (a_1 \ast b_1 \ast c) +_I  (a_2\ast b_2\ast c) \succeq_I  (a_1\ast b_2\ast c) +_I (a_2\ast b_1\ast c)
\label{eqn:calc_multi1}
\end{equation}
If addition $\ast$ is symmetric, then
\begin{equation} ( a_1+_Ib_1 +_I c)\ast ( a_2+_Ib_2 +_I c)\preceq_I ( a_1+_Ib_2 +_I c)\ast ( a_2+_Ib_1 +_I c)
\label{eqn:calc_multi2}
\end{equation}
\end{lemma}

\begin{proof}
By the distributive property of $\ast$, Eq. (\ref{eqn:calc_multi1}) can be written as
$$((a_1 \ast b_1) +_I  (a_2\ast b_2)) \ast c \succeq_I ((a_1 \ast b_2) +_I  (a_2\ast b_1)) \ast c$$
which is true by Lemma \ref{lem:calc} and the nonnegativity preserving property of $\ast$.

Similarly, Eq. (\ref{eqn:calc_multi2}) can be written as:
$$( a_1+_Ib_1)\ast ( a_2+_Ib_2) +_I c\ast (a_1+_I b_1+_I a_2+_Ib_2 +_I c) \preceq_I ( a_1+_Ib_2)\ast ( a_2+_Ib_1) +_I c\ast (a_1+_Ib_1+_Ia_2+_Ib_2 +_I c)$$
which is true by Lemma \ref{lem:calc} and the translation invariant property of the partially ordered group operation $+_I$.
\end{proof}

By choosing $g$ as the sum and $f$ as the product, or choosing $g$ as the product and $f$ as the sum, Theorem \ref{thm:newRI} along with Lemma \ref{lem:calc} can be used to prove the following result

\begin{theorem} \label{thm:newRIcalc}
Let $(I, \preceq_I, J, \preceq_J, K, \preceq_K, \ast)$ be a tuple in $\cal C$.
Let $a_1\preceq_I a_2\preceq_I\cdots \preceq_I a_n$, and  $b_1\preceq_J b_2\preceq_J \cdots \preceq_I b_n$, 
then 
$$\sum_i a_i\ast b_{n-i+1} \preceq_K \sum_i a_i\ast b_{\sigma(i)} \preceq_K \sum_i a_i\ast b_i$$
for all $\sigma\in S_n$.
If in addition
$I=J=K$, $\ast$ is symmetric,
$a_1 \succeq_I 0$ and $b_1\succeq_J 0$, then
$$ \Asterisk_i \left(a_i+b_{n-i+1}\right) \succeq_K \Asterisk_i\left( a_i+b_{\sigma(i)}\right) \succeq_K \Asterisk_i \left(a_i+b_i\right)$$
for all $\sigma\in S_n$.
\end{theorem}

Theorem \ref{thm:newRIcalc} can be used to prove the following generalized Chebyshev's sum inequality:
\begin{corollary} \label{cor:chebyshevcalc}
Let $(I, \preceq_I, J, \preceq_J, K, \preceq_K, \ast)$ be a tuple in $\cal C$.
Let $a_1\preceq_I a_2\preceq_I\cdots \preceq_I a_n$, and  $b_1\preceq_J b_2\preceq_J \cdots \preceq_I b_n$, 
then 
$$ \sum_i a_i\ast \sum_j b_{j} \preceq_K n\sum_i a_i\ast b_i.$$
\end{corollary}
\begin{proof}
$$\sum_i a_i \ast \sum_j b_j = \sum_i \sum_j a_i\ast b_j = \sum_i \sum_j a_i \ast b_{\sigma_j(i)} \preceq_K \sum_j \sum_i a_i \ast b_i \preceq_K n \sum_i a_i \ast b_i$$
where $\sigma_j(i) = (i+j \mod n)+1$. 
\end{proof}

Similarly, Theorem  \ref{thm:newRImulti} and Lemma \ref{lem:calc_multi} can be used to prove:
\begin{theorem}  \label{thm:newRImulticalc}
Let $(I, \preceq_I, I, \preceq_I, I, \preceq_I, \ast)$ be a tuple in $\cal C$.
Let $a_{ij}$ be a sequence of elements in $I$ such that for each $i$, $0\preceq_I a_{i1}\preceq_I a_{i2}\preceq_I\cdots \preceq_I a_{in}$. Then
$$\sum_i \Asterisk_j a_{j\sigma_j(i)} \preceq_I \sum_i \Asterisk_j a_{ji}$$ 
for all permutations $\sigma_j\in S_n$.
If in addition $\ast$ is symmetric, then
$$\Asterisk_i \sum_j a_{j\sigma_j(i)} \succeq_I \Asterisk_i \sum_j a_{ji}$$ 
for all permutations $\sigma_j\in S_n$.
\end{theorem}

Theorem  \ref{thm:variation} implies:
\begin{theorem}  \label{thm:variationcalc}
Let $(I, \preceq_I, I, \preceq_I, K, \preceq_K, \ast)$ be a tuple in $\cal C$ with $\ast$ symmetric.
Let $a_1\preceq_I a_2\preceq_I\cdots \preceq_I a_{2n}$ be a sequence of $2n$ elements of $I$.  
Then 
$$\sum_{i=1}^n \left(a_{i}\ast a_{2n-i+1}\right) \preceq_K \sum_{i=1}^n \left(a_{\sigma(2i-1)}\ast a_{\sigma(2i)}\right) \preceq_K \sum_{i=1}^n \left(a_{2i-1}\ast a_{2i}.\right)$$
for all $\sigma \in S_{2n}$. If in addition $I = K$ and $a_1\succeq_I 0$, then
$$\Asterisk_{i=1}^n \left(a_{2i-1}+a_{2i}\right) \preceq_I  \Asterisk_{i=1}^n \left(a_{\sigma(2i-1)}+a_{\sigma(2i)}\right) \preceq_I \Asterisk_{i=1}^n \left(a_{i}+a_{2n-i+1}\right)$$
for all $\sigma \in S_{2n}$.
\end{theorem}

Theorem \ref{thm:multiple-seq-variant} and Lemma \ref{lem:calc_multi} imply:
\begin{theorem}
\label{thm:multiple-seq-variant-calc}
Let $(I, \preceq_I, I, \preceq_I, I, \preceq_I, \ast)$ be a tuple in $\cal C$.
Let $a_{i}$ be a sequence of $kn$ elements in $I$  such that  $0\preceq_I a_{1}\preceq_I a_{2}\preceq_I\cdots \preceq_I a_{kn}$. Let $\{b_i\}$ be a permutation of $\{a_i\}$. Then
$$\sum_{i=1}^n \Asterisk_{j=1}^k b_{(i-1)k+j} \preceq_I \sum_{i=1}^n \Asterisk_{j=1}^k a_{(i-1)k+j} $$ 
If in addition $\ast$ is symmetric, then
$$\Asterisk_{i=1}^n \sum_{j=1}^k b_{(i-1)k+j} \succeq_I \Asterisk_{i=1}^n \sum_{j=1}^k a_{(i-1)k+j} $$ 
\end{theorem}

Similarly, Theorem \ref{thm:variationcalc} implies the following variation of the Chebyshev's sum inequality.
\begin{corollary}  \label{cor:variationcalcchebyshev}
Let $(I, \preceq_I, I, \preceq_I, K, \preceq_K, \ast)$ be a tuple in $\cal C$ with $\ast$ symmetric.
Let $a_1\preceq_I a_2\preceq_I\cdots \preceq_I a_{2n}$ be a sequence of $2n$ elements of $I$.  
Then 
$$ \sum_{i=1}^n a_{\sigma(i)}\ast \sum_{j=n+1}^{2n} a_{\sigma(j)} \preceq_K n\sum_{i=1}^n a_{2i-1}\ast a_{2i}.$$
for all $\sigma \in S_{2n}$.
\end{corollary}
\begin{proof}
\begin{multline*}\sum_{i=1}^n a_{\sigma(i)}\ast \sum_{j=n+1}^{2n} a_{\sigma(j)} = \sum_{i=1}^n \sum_{j=n+1}^{2n} a_{\sigma(i)}\ast a_{\sigma(j)} \\ = \sum_{i=1}^n \sum_{j=n+1}^{2n} a_{\sigma(i)} \ast a_{\sigma(\mu_j(i))} \preceq_K  \sum_{j=n+1}^{2n}  \sum_{i=1}^n a_{2i-1} \ast a_{2i} \preceq_K n \sum_{i=1}^{n} a_{2i-1} \ast a_{2i}
\end{multline*}
where $\mu_j(i) = (i+j\mod n) + n+1$. 
\end{proof}

Theorem \ref{thm:multiple-seq-variant} implies:
\begin{corollary} \label{cor:variation-multi-calc}
Let $(I, \preceq_I, I, \preceq_I, I, \preceq_I, \ast)$ be a tuple in $\cal C$ with $\ast$ symmetric.
Let $0 \preceq_I a_1\preceq_I a_2\preceq_I\cdots \preceq_I a_{kn}$ be a sequence of $kn$ elements of $I$.  
Then 
$$ \sum_{j=1}^n \Asterisk_{i=1}^k a_{\sigma((j-1)k+i)} \preceq_I \sum_{j=1}^n \Asterisk_{i=1}^k a_{(j-1)k+i}.$$
and
$$\Asterisk_{j=1}^n \sum_{i=1}^k a_{(j-1)k+i} \preceq_I \Asterisk_{j=1}^n \sum_{i=1}^k a_{\sigma((j-1)k+i)} $$
for all $\sigma \in S_{kn}$.
\end{corollary}

An analogue of Theorem \ref{thm:circular} is the following:

\begin{theorem}
Let $(I, \preceq_I, I, \preceq_I, I, \preceq_I, \ast)$ be a tuple in $\cal C$ with $\ast$ symmetric. Let $a_1\preceq_I a_2\preceq_I\cdots \preceq_I a_n$ and $V(\sigma) = a_{\sigma(1)}\ast a_{\sigma(2)} +_I a_{\sigma(2)}\ast a_{\sigma(3)}  +_I \cdots +_I a_{\sigma(n)}\ast a_{\sigma(1)} $, where $\sigma\in S_n$.
Then 
$$ V(\sigma_{m_1}) \preceq_I V(\sigma) \preceq_I V(\sigma_{m_2})$$
for all permutations $\sigma\in S_n$ where $\sigma_{m_1}$ and $\sigma_{m_2}$ are as defined in Section \ref{sec:circular}.
If in addition $a_1 \succeq_I 0$ and $W(\sigma) = \left(a_{\sigma(1)}+_{I} a_{\sigma(2)}\right)\ast \left(a_{\sigma(2)}+_I a_{\sigma(3)}\right)\ast \cdots \ast \left(a_{\sigma(n)}+_I a_{\sigma(1)}\right) $, then $W(\sigma_{m_2}) \preceq_I W(\sigma) \preceq_I W(\sigma_{m_1})$.
\end{theorem}

\subsection{Ordered inner product spaces}
Consider the case where $I=J$ is an ordered vector space $I$ with a real-valued inner product $\langle\cdot,\cdot\rangle:I\times I \rightarrow \mathbb{R}$ with corresponding partial order $\succeq$ such that
the following is true:
$$x,y \succeq 0 \Rightarrow \langle x, y\rangle \geq 0.$$ Examples of such ordered inner product spaces include $\mathbb{R}^n$, $L_2$ and $l_2$ spaces and Hermitian matrices\footnote{where the partial order is the Loewner partial order and the inner product is the Frobenius inner product $\langle A, B\rangle = tr(AB)$.}.
Then Lemma \ref{lem:calc} becomes:

\begin{lemma}\label{lem:inner}
If  $a_1 \preceq a_2$ and $b_1 \preceq b_2$, then
$$ \langle a_1,b_1\rangle + \langle a_2,b_2\rangle \geq \langle a_1,b_2\rangle + \langle a_2, b_1\rangle$$
and
$$ \langle a_1+b_1, a_2+b_2\rangle \leq \langle a_1+b_2, a_2+b_1\rangle$$
\end{lemma}

Theorem \ref{thm:newRIcalc} then becomes

\begin{theorem}
Let $a_1\preceq a_2\preceq \cdots \preceq a_n$, and $b_1\preceq b_2\preceq\cdots \preceq b_n$.
Then 
$$ \sum_i \langle a_i,b_{n-i+1}\rangle \leq \sum_i \langle a_i,b_{\sigma(i)}\rangle \leq \sum_i \langle a_i,b_i\rangle$$
for all $\sigma\in S_n$.
\end{theorem}

\subsection{Hermitian matrices}
Let us now choose $I$ and $J$ to be the set of Hermitian matrices with the Loewner partial order, i.e. $A\succeq_L B$ if $A-B$ is positive semidefinite. Since  the product of two positive semidefinite Hermitian matrices that commutes is positive semidefinite, Lemma \ref{lem:calc} implies:

\begin{lemma}\label{lem:hermitian}
Let $A_1$, $A_2$, $B_1$, $B_2$ be Hermitian matrices of the same order such that $A_i$ commutes with $B_j$ for all $i,j$.
If  $A_1 \preceq_L A_2$ and $B_1 \preceq_L B_2$, then
$$ A_1B_1 + A_2B_2 \succeq_L A_1B_2 + A_2B_1$$
If in addition $A_1$ commutes with $A_2$, then
$$ (A_1+B_1)(A_2+B_2) \preceq_L (A_2+B_1)(A_1+B_2) $$
\end{lemma}

This along with Theorem \ref{thm:newRIcalc} can be used to prove the following result which was also proved in Ref. \cite{Tie2011}.

\begin{theorem}
Let $A_1\preceq_LA_2\preceq_L\cdots \preceq_L A_n$, and  $B_1\preceq_LB_2\preceq_L\cdots \preceq_L B_n$ be Hermitian matrices of the same order such that $A_i$ commutes with $B_j$ for all $i,j$.
Then 
$$ \sum_i A_iB_{n-i+1} \preceq_L \sum_i A_iB_{\sigma(i)} \preceq_L \sum_i A_iB_i$$
for all $\sigma\in S_n$.
\end{theorem}

Similarly

\begin{theorem}
Let $0\preceq_L A_1\preceq_LA_2\preceq_L\cdots \preceq_L A_n$, and $0\preceq_L B_1\preceq_LB_2\preceq_L\cdots \preceq_L B_n$ be Hermitian matrices of the same order such that $A_i$ and $B_i$ commutes with $A_j$ and with $B_j$ for all $i,j$.
Then 
$$ \prod_i \left(A_i+B_{n-i+1}\right) \succeq_L \prod_i\left( A_i+B_{\sigma(i)}\right) \succeq_L \prod_i \left(A_i+B_i\right)$$
for all $\sigma\in S_n$.
\end{theorem}

Similarly, Theorem  \ref{thm:newRImulticalc} can be used to prove:
\begin{theorem} \label{thm:hermitianmulti}
Let $A_{ij}$ be a sequence of positive semidefinite Hermitian matrices of the same order for $1\leq i\leq k$, $1\leq j\leq n$ such that for each $i$, $A_{i1}\preceq_L A_{i2}\preceq_L\cdots \preceq_L A_{in}$ and $A_{ij}$ commutes with $A_{ml}$ for all $i\neq m$. Then
$$\sum_i \prod_j A_{j\sigma_j(i)} \preceq_L \sum_i \prod_j A_{ji}$$ 
for all permutations $\sigma_j\in S_n$.
If in addition $A_{ij}$ commutes with $A_{ml}$ for all $i,j,m,l$, then
$$\prod_i \sum_j A_{j\sigma_j(i)} \succeq_L \prod_i \sum_j A_{ji}$$ 
for all permutations $\sigma_j\in S_n$.
\end{theorem}

Theorem  \ref{thm:variationcalc} implies:
\begin{theorem} \label{them:variation-sumprod-hermitian}
Let $A_1\preceq_LA_2\preceq_L\cdots \preceq_L A_{2n}$ be a sequence of $2n$ commuting Hermitian matrices.  
Then 
$$\sum_{i=1}^n A_{i}A_{2n-i+1} \preceq_L  \sum_{i=1}^n A_{\sigma(2i-1)}A_{\sigma(2i)} \preceq_L \sum_{i=1}^n A_{2i-1}A_{2i}.$$
for all $\sigma \in S_{2n}$. If in addition $A_1\succeq_L 0$, then
$$\prod_{i=1}^n \left(A_{2i-1}+A_{2i}\right) \preceq_L  \prod_{i=1}^n \left(A_{\sigma(2i-1)}+A_{\sigma(2i)}\right) \preceq_L \prod_{i=1}^n \left(A_{i}+A_{2n-i+1}\right)$$
for all $\sigma \in S_{2n}$.
\end{theorem}

Corollary \ref{cor:variation-multi-calc} implies:
\begin{corollary} \label{cor:variation-prodsum-multiple-general}
Let $0 \preceq_L A_1\preceq_L A_2\preceq_L\cdots \preceq_L A_{kn}$ be a sequence of $kn$ commuting Hermitian matrices.  
Then 
$$ \sum_{j=1}^n \prod_{i=1}^k A_{\sigma((j-1)k+i)} \preceq_L \sum_{j=1}^n \prod_{i=1}^k A_{(j-1)k+i}$$
and 
$$\prod_{j=1}^n \sum_{i=1}^k A_{(j-1)k+i} \preceq_L \prod_{j=1}^n \sum_{i=1}^k A_{\sigma((j-1)k+i)} $$
for all $\sigma \in S_{kn}$.
\end{corollary}

For both the Kronecker product $\otimes$ and Hadamard product $\odot$, the product of two positive semidefinite Hermitian matrices is Hermitian and positive semidefinite. In addition, the Hadamard product is a symmetric operator. 
Lemma \ref{lem:calc} then implies the following:

\begin{lemma}\label{lem:kronhadamard}
Let $A_1$, $A_2$,  $B_1$, $B_2$ be Hermitian matrices. If  $A_1 \preceq_L A_2$ and $B_1 \preceq_L B_2$, then
$$ A_1\otimes B_1 + A_2\otimes B_2 \succeq_L A_1\otimes B_2 + A_2\otimes B_1$$
If in addition $A_i$ and $B_i$ are of the same order, then
$$ A_1\odot B_1 + A_2\odot B_2 \succeq_L A_1\odot B_2 + A_2\odot B_1$$
$$ (A_1+B_1)\odot (A_2+B_2) \preceq_L (A_2+B_1)\odot (A_1+B_2) $$
\end{lemma}

This allows us to prove the following series of results:

\begin{theorem} \label{thm:kronhadamard1}
Let $A_1\preceq_LA_2\preceq_L\cdots \preceq_L A_n$, and $B_1\preceq_LB_2\preceq_L\cdots \preceq_L B_n$ be Hermitian matrices.
Then 
$$ \sum_i \left(A_i\otimes B_{n-i+1}\right) \preceq_L \sum_i \left(A_i\otimes B_{\sigma(i)}\right) \preceq_L \sum_i \left(A_i\otimes B_i\right)$$
for all $\sigma\in S_n$.
If in addition $A_i$ and $B_i$ are of the same order,  then
$$ \sum_i \left(A_i\odot B_{n-i+1}\right) \preceq_L \sum_i \left(A_i\odot B_{\sigma(i)}\right) \preceq_L \sum_i \left(A_i\odot B_i\right)$$
for all $\sigma\in S_n$.
\end{theorem}

Theorem \ref{thm:kronhadamard1} was also shown in Ref. \cite{Tie2011}.

\begin{theorem}
Let $0\preceq_L A_1\preceq_LA_2\preceq_L\cdots \preceq_L A_n$, and $0\preceq_L B_1\preceq_LB_2\preceq_L\cdots \preceq_L B_n$ be Hermitian matrices of the same order.
Then 
$$ \bigodot_i \left(A_i+B_{n-i+1}\right) \succeq_L \bigodot_i\left( A_i+B_{\sigma(i)}\right) \succeq_L \bigodot_i \left(A_i+B_i\right)$$
for all $\sigma\in S_n$.
\end{theorem}

\begin{theorem} 
Let $A_{ij}$ be a sequence of positive semidefinite Hermitian matrices such that for each $i$, $A_{i1}\preceq_L A_{i2}\preceq_L\cdots \preceq_L A_{in}$. Then
$$\sum_i \bigotimes_j A_{j\sigma_j(i)} \preceq_L \sum_i \bigotimes_j A_{ji},$$ 
\end{theorem}

\begin{theorem} 
Let $A_{ij}$ be a sequence of positive semidefinite Hermitian matrices of the same order for $1\leq i\leq k$, $1\leq j\leq n$ such that for each $i$, $A_{i1}\preceq_L A_{i2}\preceq_L\cdots \preceq_L A_{in}$. Then
$$\sum_i \bigodot_j A_{j\sigma_j(i)} \preceq_L \sum_i \bigodot_j A_{ji}$$ 
and
$$\bigodot_i \sum_j A_{j\sigma_j(i)} \succeq_L \bigodot_i \sum_j A_{ji}$$ 
for all permutations $\sigma_j\in S_n$.
\end{theorem}

\begin{theorem} 
Let $A_1\preceq_LA_2\preceq_L\cdots \preceq_L A_{2n}$ be a sequence of $2n$ Hermitian matrices.  
Then 
$$\sum_{i=1}^n \left(A_{i}\odot A_{2n-i+1}\right) \preceq_L  \sum_{i=1}^n \left(A_{\sigma(2i-1)}\odot A_{\sigma(2i)}\right) \preceq_L \sum_{i=1}^n \left(A_{2i-1}\odot A_{2i}\right).$$
for all $\sigma \in S_{2n}$. If in addition $A_1\succeq_L 0$, then
$$\bigodot_{i=1}^n \left(A_{2i-1}+A_{2i}\right) \preceq_L  \bigodot_{i=1}^n \left(A_{\sigma(2i-1)}+A_{\sigma(2i)}\right) \preceq_L \bigodot_{i=1}^n \left(A_{i}+A_{2n-i+1}\right)$$
for all $\sigma \in S_{2n}$.
\end{theorem}

\begin{corollary} 
Let $0 \preceq_L A_1\preceq_L A_2\preceq_L\cdots \preceq_L A_{kn}$ be a sequence of $kn$ Hermitian matrices.  
Then 
$$ \sum_{j=1}^n \bigodot_{i=1}^k A_{\sigma((j-1)k+i)} \preceq_L \sum_{j=1}^n \bigodot_{i=1}^k A_{(j-1)k+i}$$
$$\bigodot_{j=1}^n \sum_{i=1}^k A_{(j-1)k+i} \preceq_L \bigodot_{j=1}^n \sum_{i=1}^k A_{\sigma((j-1)k+i)} $$
for all $\sigma \in S_{kn}$.
\end{corollary}

\section{Conclusions}
We consider several variants and generalizations of the rearrangement inequality for which we can generalize to multiple sequences and find both the set of permutations that maximizes or minimizes the sum of products or product of sums of terms and where the permutation can be chosen across sequences. We also study rearrangement inequalities beyond real numbers where the elements are vectors, matrices or functions.

\end{document}